\documentclass[11pt,reqno]{amsart}
\usepackage{amsmath,amssymb}
\usepackage{mathrsfs}
\usepackage{amsfonts}
\usepackage[dvipsnames]{xcolor}
\usepackage{graphicx}
\usepackage{float}  
\usepackage{cite}
\usepackage{cases}
\usepackage{indentfirst}
\allowdisplaybreaks

\newtheorem{theorem}{Theorem}[section]
\newtheorem{lemma}[theorem]{Lemma}
\newtheorem{corollary}[theorem]{Corollary}

\newtheorem{remark}[theorem]{Remark}

\numberwithin{equation}{section}

\makeatletter
\usepackage[colorlinks=true,citecolor=blue,linkcolor=blue]{hyperref}

\begin{document}
\author{Yasheng Lyu}

\title{\textbf{Second-order estimates for degenerate complex $k$-Hessian and Christoffel-Minkowski equations} }

\address{School of Mathematics and Statistics, Xi'an Jiaotong University, Xi'an, Shaanxi 710049, People's Republic of China}

\email{lvysh21@stu.xjtu.edu.cn}

\begin{abstract}
It is known that the complex $k$-Hessian equation admits almost $C^{1,1}$ regularity (i.e., $\sup\Delta u<\infty$) and the Christoffel-Minkowski equation admits $C^{1,1}$ regularity under the sharp degenerate condition $f^{1/(k-1)}\in C^{1,1}$ for a nonnegative right-hand side $f$.
Assuming instead the alternative sharp degenerate condition $f^{3/(2k-2)}\in C^{2,1}$, we prove almost $C^{1,1}$ regularity for the complex $k$-Hessian equation when $k\geq5$ and $C^{1,1}$ regularity for the Christoffel-Minkowski equation.
The argument deeply exploits various concavity properties of the operators under the stronger regularity assumption on $f$.
\end{abstract}


 \date{}
\maketitle

\pagestyle{myheadings}
\markboth{$~$ \hfill{\uppercase{Yasheng Lyu}}\hfill $~$}{$~$ \hfill{\uppercase{Degenerate complex $k$-Hessian and Christoffel-Minkowski equations}}\hfill $~$}

\section{Introduction}

The complex Monge-Amp\`ere equation plays a central role in complex geometry.
In a celebrated article, Yau \cite{Yau1978} (and independently Aubin \cite{Aubin1978}) solved the Calabi conjecture by studying the complex Monge-Amp\`ere equation on compact K\"ahler manifolds.

Let $(M,\omega)$ be a compact K\"ahler manifold of complex dimension $n\geq2$, where $\omega$ is the K\"ahler form on $M$.
The complex Monge-Amp\`ere equation on $M$ takes the form
\begin{equation}\label{eqn1.1}
(dd^{c}u+\omega)^{n}=f\omega^{n},
\end{equation}
where $d=\partial+\bar{\partial}$, $d^{c}=\sqrt{-1}(\bar{\partial}-\partial)$, and $\omega^{n}=\omega\wedge\cdots\wedge\omega$.
The compatibility condition for \eqref{eqn1.1} is 
\[
\int_{M}f\omega^{n}=\int_{M}\omega^{n}.
\]
We consider the complex $k$-Hessian equation
\begin{equation}\label{eqn1.2}
\binom{n}{k}(dd^{c}u+\omega)^{k}\wedge\omega^{n-k}=f\omega^{n},
\end{equation}
where $2\leq k\leq n$.
The compatibility condition for \eqref{eqn1.2} is 
\begin{equation}\label{eqn1.3}
\int_{M}f\omega^{n}=\binom{n}{k}\int_{M}\omega^{n}.
\end{equation}
As usual, we normalize $u$ by 
\begin{equation}\label{eqn1.4}
\int_{M}u\omega^{n}=0.
\end{equation}
Following Hou-Ma-Wu \cite{Hou2010}, a function $u\in C^{2}(M)$ is called \emph{$k$-admissible} if 
\[
\frac{(dd^{c}u+\omega)^{j}\wedge\omega^{n-j}}{\omega^{n}}>0\quad \text{for all}\ 1\leq j\leq k.
\]

For an $n\times n$ Hermitian matrix $A$, let $\sigma_{k}(A)$ denote the $k$-th elementary symmetric function of its eigenvalues,
\[
\sigma_{k}(A):=\sigma_{k}(\lambda(A)):=\sum_{1\leq i_{1}<\cdots<i_{k}\leq n}\lambda_{i_{1}}\cdots\lambda_{i_{k}} 
\]
for $k=1,2,\dots,n$, where $\lambda(A)=(\lambda_{1},\dots,\lambda_{n})$ denotes the (real) eigenvalues of $A$. 
The G{\aa}rding cone is defined by 
\[
\Gamma_{k}:=\left\{\lambda\in\mathbb{R}^{n}:\ \sigma_{j}(\lambda)>0\ \text{for all}\ 1\leq j\leq k\right\}.
\]
The operator $\sigma_{k}^{1/k}$ is elliptic and concave in $\Gamma_{k}$.
Since $d\omega=0$, locally there exists a smooth plurisubharmonic function $g$ such that
\[
\omega=dd^{c}g=2\sum_{i,j=1}^{n}g_{i\bar{j}}\sqrt{-1}dz_{i}\wedge d\bar{z}_{j}.
\]
If $u$ is smooth, then equation \eqref{eqn1.2} can be written as 
\begin{equation}\label{eqn1.5}
\sigma_{k}(W)=f,
\end{equation}
where 
\[
W:=\Big(g^{i\bar{j}}\Big)\big(w_{i\bar{j}}\big),\quad \Big(g^{i\bar{j}}\Big):=\big(g_{i\bar{j}}\big)^{-1},\quad w:=u+g.
\]
Moreover, $u$ is $k$-admissible if and only if $\lambda(W)\in\Gamma_{k}$.

The equation \eqref{eqn1.2} is called \emph{degenerate} if the nonnegative function $f$ is allowed to vanish at some points on $M$, and \emph{non-degenerate} if $\inf_{M}f>\varepsilon$ for some $\varepsilon>0$. 
We focus on the degenerate case, where the optimal exponent in the regularity assumptions on $f$ becomes delicate near $\{f=0\}$.
The condition $f^{1/k}\in C^{1,1}$ arises naturally from the concavity of the operator $\sigma_{k}^{1/k}$, but it is not optimal for the existence of $C^{1,1}$ solutions.
In fact, the optimal exponent is closely connected to the concavity structure.
The conditions
\[
f^{\frac{1}{k-1}}\in C^{1,1}\quad \text{and}\quad f^{\frac{3}{2(k-1)}}\in C^{2,1}
\] 
are sharp, as demonstrated by counterexamples in Wang \cite{Wang1995}, Ivochkina-Trudinger-Wang \cite{Wang2004}, Pli\'s \cite{Plis2005}, and Dinew-Pli{\'s}-Zhang \cite{Zhang2019}. 
The condition $f^{1/(k-1)}\in C^{1,1}$ has been extensively investigated; see, for instance, Guan-Li \cite{Guan1997a}, Guan \cite{Guan1997}, Guan-Trudinger-Wang \cite{Guan1999}, Li-Weinstein \cite{Li1999}, Dong \cite{Dong2006}, Guan-Zhang \cite{Guan2021}, Jiao-Wang \cite{Jiao2022,Jiao2024}, Jiao-Sun \cite{Jiao2022b}, Jiao-Jiao \cite{Jiao2024b}, and references therein.
To establish the existence of $C^{1,1}$ solutions, a common strategy is to derive $C^{2}$ a priori estimates for non-degenerate problems with constants independent of $\inf_{M}f$.
One then obtains $C^{1,1}$ solutions in the degenerate case by an approximation argument. 

We establish the following second-order a priori estimate for \eqref{eqn1.2}.

\begin{theorem}\label{thm1.1}
Let $u\in C^{4}(M)$ be a $k$-admissible solution of \eqref{eqn1.2}.
Assume that \eqref{eqn1.3} and \eqref{eqn1.4} hold. 
Let $k\geq5$ and $f^{3/(2k-2)}\in C^{2,1}(M)$ with $\inf_{M}f>0$.
Then 
\begin{equation}\label{eqn1.6}
\sup_{M}\Delta u\leq C\left(\sup_{M}|\nabla u|^{2}+1\right),
\end{equation}
where $C>0$ depends on $\|u\|_{C^{0}(M)}$, $\|f^{3/(2k-2)}\|_{C^{2,1}(M)}$, and the geometry of $(M,\omega)$, but is independent of $\inf_{M}f$.
\end{theorem}

The $C^{0}$ estimate for \eqref{eqn1.2} was established by Hou-Ma-Wu \cite{Hou2010} via Yau's Moser iteration.
As pointed out in \cite{Hou2010}, an estimate of the form \eqref{eqn1.6} is well suited to blow-up analysis.
In particular, using a Liouville theorem together with the blow-up method, Dinew-Ko{\l}odziej \cite{Dinew2017} derived a gradient estimate for \eqref{eqn1.2} based on an estimate of the form \eqref{eqn1.6}.
Consequently, combining Theorem \ref{thm1.1} with the aforementioned results, one obtains an \emph{almost} $C^{1,1}$ solution of \eqref{eqn1.2}, where a solution $u$ is called \emph{almost} $C^{1,1}$ if 
\[
\sup_{M}\Delta u\leq C.
\]
This is equivalent to the boundedness of the mixed complex derivatives $u_{i\bar{j}}$ for $i,j=1,\dots,n$.
Every almost $C^{1,1}$ solution belongs to $W^{2,p}$ for all $p<\infty$, and hence to $C^{1,\alpha}$ for all $\alpha<1$.
However, it need not belong to $W^{2,\infty}=C^{1,1}$.

\begin{corollary}
Assume that \eqref{eqn1.3} and \eqref{eqn1.4} hold, and let $f\geq0$ with $f^{3/(2k-2)}\in C^{2,1}(M)$.
Let $k\geq5$.
Then \eqref{eqn1.2} admits a unique $k$-admissible almost $C^{1,1}$ solution.
\end{corollary}

\par
\vspace{2mm}
The problem of finding a convex hypersurface with the $k$-th elementary symmetric function of the principal radii prescribed on its outer normals is often called the Christoffel-Minkowski problem.
It corresponds to finding a regular convex solution of 
\begin{equation}\label{eqn1.7}
\sigma_{k}((u_{ij}+u\delta_{ij}))=f\quad \text{on}\ \mathbb{S}^{n},
\end{equation}
with the positive definite condition
\begin{equation}\label{eqn1.8}
(u_{ij}+u\delta_{ij})>0\quad \text{on}\ \mathbb{S}^{n},
\end{equation}
where $u_{ij}$ are the second-order covariant derivatives with respect to any orthonormal frame $\{e_{1},e_{2},\dots,e_{n}\}$ on $\mathbb{S}^{n}$, and $\delta_{ij}$ is the standard Kronecker symbol.
We refer to \eqref{eqn1.7} as the Christoffel-Minkowski equation.
A necessary condition for solvability of \eqref{eqn1.7}--\eqref{eqn1.8} is
\begin{equation}\label{eqn1.9}
\int_{\mathbb{S}^{n}}x_{i}f(x)\ dx=0\quad \text{for}\ i=1,2,\dots,n+1.
\end{equation}

In the case $k=1$, \eqref{eqn1.7}--\eqref{eqn1.8} reduces to the Christoffel problem, which was solved by Firey \cite{Firey1967,Firey1968} and Berg \cite{Berg1969}.
In the case $k=n$, \eqref{eqn1.7}--\eqref{eqn1.8} corresponds to the Minkowski problem, and this case was settled through the works of Minkowski \cite{Minkowski1897}, Alexandrov \cite{Alexandrov1937}, Lewy \cite{Lewy1938}, Nirenberg \cite{Nirenberg1953}, Pogorelov \cite{Pogorelov1978}, and Cheng-Yau \cite{Yau1976}.
For the intermediate Christoffel-Minkowski problems $(2\leq k\leq n-1)$, solvability was established by Guan-Ma \cite{Guan2003}.
The equation \eqref{eqn1.7} is called \emph{degenerate} if the nonnegative function $f$ is allowed to vanish at some points on $\mathbb{S}^{n}$, and \emph{non-degenerate} if $\inf_{\mathbb{S}^{n}}f>\varepsilon$ for some $\varepsilon>0$. 
A function $u\in C^{1,1}(\mathbb{S}^{n})$ is called \emph{$k$-admissible} if 
\[
\lambda((u_{ij}+u\delta_{ij}))\in\Gamma_{k}\quad \text{a.e. on}\ \mathbb{S}^{n}.
\]
Guan-Zhang \cite{Guan2021} established the existence of $k$-admissible solutions in the degenerate case under the sharp condition $f^{1/(k-1)}\in C^{1,1}(\mathbb{S}^{n})$.

We establish the following existence theorem for \eqref{eqn1.7} in the degenerate case.

\begin{theorem}\label{thm1.3}
Let $f\geq0$ and $f^{3/(2k-2)}\in C^{2,1}(\mathbb{S}^{n})$.
Then \eqref{eqn1.7} admits a unique $k$-admissible $C^{1,1}(\mathbb{S}^{n})$ solution up to translations.
\end{theorem}

\par
\vspace{2mm}
\textbf{Notation.}
Throughout the paper, we suppress the dependence of constants on the dimension $n$ and the parameter $k$, and a constant depending only on $n$ and $k$ is called \emph{universal}. 
\par
\vspace{2mm}

The rest of the paper is organized as follows. 
Section 2 collects preliminary results.
Section 3 establishes the second-order estimate for the complex $k$-Hessian equation, proving Theorem \ref{thm1.1}.
Section 4 is devoted to the existence of $k$-admissible solutions to the Christoffel-Minkowski equation, proving Theorem \ref{thm1.3}.
The Appendix (Section 5) contains auxiliary lemmas.

\section{Preliminaries}

We adopt the following conventions: $\sigma_{0}(\lambda)=1$; $\sigma_{k}(\lambda)=0$ for $k<0$ and $k>n$; 
\[
\sigma_{k;i_{1}\cdots i_{j}}(\lambda)=\sigma_{k}(\lambda|i_{1}\cdots i_{j})=\sigma_{k}(\lambda)\big|_{\lambda_{i_{1}}=\cdots=\lambda_{i_{j}}=0};
\]
moreover, $\sigma_{k;i_{1}\cdots i_{j}}=0$ if $i_{r}=i_{s}$ for some $1\leq r<s\leq j$. 
Some fundamental properties of $\sigma_k$ are listed below.

\begin{lemma}[\cite{Hardy1952,Wang2009}]\label{thm2.1}
For any $k\in\{2,3,\dots,n\}$ and $\lambda\in\Gamma_{k}$, we have
\[
\sigma_{k}^{\frac{1}{k}}(\lambda)\leq C_{1}\sigma_{k-1}^{\frac{1}{k-1}}(\lambda);
\]
\[
\sigma_{k-1}(\lambda)\geq C_{2}\sigma_{1}^{\frac{1}{k-1}}(\lambda)\sigma_{k}^{1-\frac{1}{k-1}}(\lambda);
\]
\[
\sum_{i=1}^{n}\sigma_{k-1;i}(\lambda)=(n-k+1)\sigma_{k-1}(\lambda),
\]
where $C_{1},C_{2}>0$ are universal.
\end{lemma}

For an $n\times n$ Hermitian matrix $A$, define 
\[
F(A):=\sigma_{k}^{\frac{1}{k}}(\lambda(A)),\quad F^{i\bar{j}}(A):=\frac{\partial F}{\partial A_{i\bar{j}}},\quad F^{i\bar{j},s\bar{t}}(A):=\frac{\partial^{2}F}{\partial A_{i\bar{j}}A_{s\bar{t}}},\quad \operatorname{tr}F^{i\bar{j}}:=\sum_{i=1}^{n}F^{i\bar{i}}.
\]
It follows from Lemma \ref{thm2.1} that  
\begin{equation}\label{eqn2.1}
\operatorname{tr}F^{i\bar{j}}=\frac{1}{k}\sigma_{k}^{\frac{1}{k}-1}\sum_{i=1}^{n}\sigma_{k-1;i}=\frac{n-k+1}{k}\sigma_{k}^{\frac{1}{k}-1}\sigma_{k-1}\geq C\sigma_{1}^{\frac{1}{k-1}}\sigma_{k}^{-\frac{1}{k(k-1)}},
\end{equation}
where $C>0$ is universal.
Furthermore, when $A$ is diagonal we have 
\begin{equation}\label{eqn2.2}
F^{i\bar{j},s\bar{t}}(A)=
\begin{cases}
\frac{\partial^{2}}{\partial\lambda_{i}\partial\lambda_{s}}\big(\sigma_{k}^{1/k}\big),\quad\text{if}\ i=j,s=t;\\
-\frac{1}{k}\sigma_{k}^{1/k-1}\sigma_{k-2;is},\quad\text{if}\ i\neq j,s=j,\ \text{and}\ t=i;\\
0,\quad\text{otherwise}.
\end{cases}
\end{equation}

We will use the following lemma from Guan-Li-Li \cite[Lemma 3.2]{Guan2012}, which follows from the concavity of 
\[
\bigg(\frac{\sigma_{k}}{\sigma_{1}}\bigg)^{\frac{1}{k-1}}\quad \text{in}\ \Gamma_{k}.
\]

\begin{lemma}\label{thm2.2}
For any $2\leq k\leq n$, we have
\[
\frac{\sigma_{k}^{ij,st}A_{ij}A_{st}}{\sigma_{k}}\leq \frac{k-2}{k-1}\Bigg(\frac{\sigma_{k}^{ij}A_{ij}}{\sigma_{k}}\Bigg)^{2}+\frac{2}{k-1}\frac{\sigma_{k}^{ij}A_{ij}}{\sigma_{k}}\frac{\sigma_{1}^{ij}A_{ij}}{\sigma_{1}}-\frac{k}{k-1}\Bigg(\frac{\sigma_{1}^{ij}A_{ij}}{\sigma_{1}}\Bigg)^{2},
\]
where all operators are evaluated at a symmetric matrix $W$ with $\lambda(W)\in\Gamma_{k}$, $A_{ij}$ are the entries of a matrix $A$, and we use the summation convention.
\end{lemma}

The following lemma is a basic estimate for real functions on $\mathbb{R}^{n}$; see \cite[Lemma 2.1]{Lyu2025}.

\begin{lemma}[\cite{Lyu2025}]\label{thm2.3}
Let $\Omega\Subset\Omega_{0}$ and let $h>0$ in $\Omega_{0}$.
Then the following hold.  

$(\romannumeral1)$ If $h\in C^{1,1}(\overline{\Omega_{0}})$, then 
\[
\frac{|\nabla h(x)|^{2}}{h(x)}\leq K,\quad \forall x\in\overline{\Omega},
\]
where $K>0$ depends on $\|h\|_{C^{1,1}(\overline{\Omega_{0}})}$ and $\operatorname{dist}(\Omega,\partial\Omega_{0})$, but is independent of $\inf_{\Omega_{0}}h$.

$(\romannumeral2)$ If $h\in C^{2,1}(\overline{\Omega_{0}})$, then for any $\alpha<1/2$, 
\begin{equation}\label{eqn2.3}
\partial_{ee}h(x)-\alpha\frac{|\partial_{e}h(x)|^{2}}{h(x)}\geq-Kh^{\frac{1}{3}}(x),\quad \forall x\in\overline{\Omega},\ e\in\mathbb{S}^{n-1},
\end{equation}
where $K>0$ depends on $\alpha$, $\|h\|_{C^{2,1}(\overline{\Omega_{0}})}$, and  $\operatorname{dist}(\Omega,\partial\Omega_{0})$, but is independent of $\inf_{\Omega_{0}}h$.
\end{lemma}

\begin{remark}
In general, one cannot weaken the condition $h>0$ in $\Omega_{0}$ to $h>0$ in $\Omega$. 
A counterexample was given by B{\l}ocki \cite{Blocki2003}: the function $h(x):=x$ on $[0,1]$ is smooth, but $\sqrt{h}\notin C^{0,1}([0,1])$.  
However, the condition $h>0$ in $\Omega$ is sufficient to control derivatives of $h$ in directions tangential to $\partial\Omega$. 
\end{remark}

\begin{remark}
In general, it is impossible to improve the exponent $\alpha$ to any $\alpha>1/2$. 
Consider the family of functions $h_{\beta}(x):=(x+\beta^{1/2})^{2}$, where $\beta>0$ is a parameter.
It is clear that $h_{\beta}>0$ and $h_{\beta}\in C^{\infty}(\mathbb{R})$. 
Computing \eqref{eqn2.3} at $x=0$ and letting $\beta\rightarrow0$ yields  $\alpha\leq1/2$. 
\end{remark}

\section{Complex $k$-Hessian equation}

We establish the following lemma for real-valued functions on a compact K\"ahler manifold, whose proof is postponed to the Appendix.

\begin{lemma}\label{thm3.1}
Let $(M,\omega)$ be a compact K\"ahler manifold of complex dimension $n\geq2$, and let $\tilde{f}$ be a positive function on $M$.
Then the following hold.

$(\romannumeral1)$ If $\tilde{f}\in C^{1,1}(M)$, then 
\[
\frac{\big|\nabla_{e}\tilde{f}(x)\big|^{2}}{\tilde{f}(x)}\leq K,\quad \forall x\in M,\ \forall e\in T_{x}M,\ |e|_{\omega}=1,
\]
where $K>0$ depends on $\big\|\tilde{f}\big\|_{C^{1,1}(M)}$ and the geometry of $(M,\omega)$, but is independent of $\inf_{M}\tilde{f}$.

$(\romannumeral2)$ If $\tilde{f}\in C^{2,1}(M)$, then for any $\alpha<1/2$,
\begin{equation}\label{eqn3.21}
\nabla_{e}\nabla_{\bar{e}}\tilde{f}(x)-\alpha\frac{\big|\nabla_{e}\tilde{f}(x)\big|^{2}}{\tilde{f}(x)}\geq-K\tilde{f}^{\frac{1}{3}}(x),\quad \forall x\in M,\ \forall e\in T_{x}M,\ |e|_{\omega}=1,
\end{equation}
where $K>0$ depends on $\alpha$, $\big\|\tilde{f}\big\|_{C^{2,1}(M)}$, and the geometry of $(M,\omega)$, but is independent of $\inf_{M}\tilde{f}$.
\end{lemma}

Both $F=\sigma_{k}^{1/k}$ and $\log\sigma_{k}$ are concave in $\Gamma_{k}$, and the concavity of $\log\sigma_{k}$ follows from that of $F$.
In the proof of Theorem \ref{thm1.1}, we differentiate the equation
\[
F(W)=f^{\frac{1}{k}}.
\]
It is important to work with $F$ rather than with $\log\sigma_{k}(W)=\log f$, since our argument uses the inequality \eqref{eqn3.21}, where the admissible range of the parameter $\alpha$ is delicate.
We are now ready to prove Theorem \ref{thm1.1}.
The argument is inspired by Chou-Wang \cite{Wang2001}, Hou-Ma-Wu \cite{Hou2010}, and Dinew-Pli{\'s}-Zhang \cite{Zhang2019}.

\par
\vspace{2mm}
\begin{proof}[Proof of Theorem \ref{thm1.1}]
Denote 
\[
\alpha:=\sup_{M}|u|+1,\quad \beta:=\sup_{M}|\nabla u|^{2}+1,\quad \gamma:=\sup_{M}\left|\inf_{\eta,\zeta}R_{\eta\bar{\eta}\zeta\bar{\zeta}}\right|.
\]
Construct the auxiliary function 
\[
G(x,\xi):=\log\big(1+u_{i\bar{j}}\xi\bar{\xi}^{j}\big)+\varphi\big(|\nabla u|^{2}\big)+\psi(u)
\]
for $x\in M$ and unit vectors $\xi\in T_{x}^{1,0}M$, where 
\[
\varphi(t):=-\frac{1}{2}\log\bigg(1-\frac{t}{2\beta}\bigg)\quad \text{and}\quad \psi(t):=-3\alpha(2\gamma+1)\log\bigg(1+\frac{t}{2\alpha}\bigg).
\]
We will use the following properties of $\varphi$ and $\psi$, evaluated at $|\nabla u|^{2}$ and $u$, respectively:
\begin{equation}\label{eqn3.1}
\begin{cases}
|\varphi|\leq\log2\\
(4\beta)^{-1}\leq\varphi'\leq(2\beta)^{-1}\\
\varphi''-2\varphi'^{2}\geq0
\end{cases}
\end{equation}
and 
\begin{equation}\label{eqn3.2}
\begin{cases}
|\psi|\leq3\alpha(2\gamma+1)\log2\\
2\gamma+1\leq-\psi'\leq3(2\gamma+1)\\
\psi''-2\delta\psi'^{2}\geq0,
\end{cases}
\end{equation}
where $\delta:=1/[6\alpha(2\gamma+1)]$.
Since $M$ is compact, $G$ attains its maximum at some $x_{0}\in M$ and some unit vector $\xi_{0}\in T_{x_{0}}^{1,0}M$.
Since $M$ is K\"ahler, we choose a local normal coordinate system near $x_{0}$ such that 
\[
g_{i\bar{j}}(x_{0})=\delta_{ij},\quad \frac{\partial g_{i\bar{j}}}{\partial z^{l}}(x_{0})=0,\quad u_{i\bar{j}}(x_{0})=\delta_{ij}u_{i\bar{i}}(x_{0}),\quad u_{1\bar{1}}(x_{0})\geq\cdots\geq u_{n\bar{n}}(x_{0}).
\]
Then $W$ is diagonal at $x_{0}$, and hence so is the positive definite matrix $(F^{i\bar{j}}(W))$.
Without loss of generality, assume $u_{1\bar{1}}(x_{0})\gg1$.
Let $\lambda=(\lambda_{1},\dots,\lambda_{n})$ be the eigenvalues of $W(x_{0})$.
It follows that 
\[
\lambda_{i}=u_{i\bar{i}}(x_{0})+1\quad \text{and}\quad \lambda_{1}\geq\cdots\geq\lambda_{n}.
\]

By the definition of $H$ and the above construction, we see that $\xi_{0}$ must coincide with $\partial/\partial z^{1}$ at $x_{0}$.
We extend $\xi_{0}$ to a smooth unit vector field in a neighborhood of $x_{0}$ by setting 
\[
\xi_{0}=g^{-\frac{1}{2}}\frac{\partial}{\partial z^{1}}.
\]
Note that the function
\[
H(x):=G(x,\xi_{0})=\log\big(1+g_{1\bar{1}}^{-1}u_{1\bar{1}}\big)+\varphi\big(|\nabla u|^{2}\big)+\psi(u)
\]
is well-defined in a neighborhood of $x_{0}$, and $H$ achieves its maximum at $x_{0}$.
At $x_{0}$ we have
\begin{equation}\label{eqn3.3}
0=H_{i}=\frac{u_{1\bar{1}i}}{1+u_{1\bar{1}}}+\varphi'u_{i}u_{\bar{i}i}+\varphi'\sum_{s=1}^{n}u_{si}u_{\bar{s}}+\psi'u_{i}
\end{equation}
for $i=1,\dots,n$, and 
\begin{align}\label{eqn3.4}
0&\geq\sum_{i,j=1}^{n}F^{i\bar{j}}H_{i\bar{j}}=\sum_{i=1}^{n}F^{i\bar{i}}H_{i\bar{i}}\nonumber\\
&=\sum_{i=1}^{n}\frac{F^{i\bar{i}}u_{1\bar{1}i\bar{i}}}{1+u_{1\bar{1}}}-\sum_{i=1}^{n}\frac{F^{i\bar{i}}|u_{1\bar{1}i}|^{2}}{(1+u_{1\bar{1}})^{2}}+\varphi''\sum_{i=1}^{n}F^{i\bar{i}}\left|u_{i}u_{i\bar{i}}+\sum_{s=1}^{n}u_{si}u_{\bar{s}}\right|^{2}\nonumber\\
&\hspace{4.5mm}+\varphi'\sum_{i=1}^{n}F^{i\bar{i}}u_{i\bar{i}}^{2}+\varphi'\sum_{i,s=1}^{n}F^{i\bar{i}}|u_{si}|^{2}+\varphi'\sum_{i,s=1}^{n}F^{i\bar{i}}(u_{si\bar{i}}u_{\bar{s}}+u_{\bar{s}i\bar{i}}u_{s})\nonumber\\
&\hspace{4.5mm}+\psi''\sum_{i=1}^{n}F^{i\bar{i}}|u_{i}|^{2}+\psi'\sum_{i=1}^{n}F^{i\bar{i}}u_{i\bar{i}}\nonumber\\
&\geq\sum_{i=1}^{n}\frac{F^{i\bar{i}}u_{1\bar{1}i\bar{i}}}{1+u_{1\bar{1}}}-\sum_{i=1}^{n}\frac{F^{i\bar{i}}|u_{1\bar{1}i}|^{2}}{(1+u_{1\bar{1}})^{2}}+\varphi''\sum_{i=1}^{n}F^{i\bar{i}}\left|u_{i}u_{i\bar{i}}+\sum_{s=1}^{n}u_{si}u_{\bar{s}}\right|^{2}\\
&\hspace{4.5mm}+\varphi'\sum_{i=1}^{n}F^{i\bar{i}}u_{i\bar{i}}^{2}+\varphi'\sum_{i,s=1}^{n}F^{i\bar{i}}(u_{si\bar{i}}u_{\bar{s}}+u_{\bar{s}i\bar{i}}u_{s})+\psi''\sum_{i=1}^{n}F^{i\bar{i}}|u_{i}|^{2}+\psi'\sum_{i=1}^{n}F^{i\bar{i}}u_{i\bar{i}}.\nonumber
\end{align}
By commuting the covariant derivatives, we have 
\[
u_{i\bar{j}s}=u_{is\bar{j}}-\sum_{a}u_{a}R^{a}_{is\bar{j}}\quad \text{and}\quad  u_{i\bar{j}s\bar{t}}=u_{s\bar{t}i\bar{j}}+\sum_{a}u_{a\bar{j}}R^{a}_{is\bar{t}}-\sum_{b}u_{b\bar{t}}R^{b}_{is\bar{j}}.
\]
At $x_{0}$ we have
\begin{equation}\label{eqn3.5}
\sum_{i=1}^{n}F^{i\bar{i}}u_{si\bar{i}}=\Big(f^{\frac{1}{k}}\Big)_{s}+\sum_{i,t=1}^{n}u_{t}F^{i\bar{i}}R_{s\bar{t}i\bar{i}}\quad \text{for}\ s=1,\dots,n,
\end{equation}
\begin{equation}\label{eqn3.6}
\sum_{i=1}^{n}F^{i\bar{i}}u_{1\bar{1}i\bar{i}}=\Big(f^{\frac{1}{k}}\Big)_{1\bar{1}}-\sum_{i,j,s,t=1}^{n}F^{i\bar{j},s\bar{t}}u_{i\bar{j}1}u_{s\bar{t}\bar{1}}+\sum_{i=1}^{n}F^{i\bar{i}}(u_{1\bar{1}}-u_{i\bar{i}})R_{1\bar{1}i\bar{i}}.
\end{equation}
It follows from \eqref{eqn3.1} and \eqref{eqn3.5} that
\begin{align}\label{eqn3.7}
\varphi'\sum_{i,s=1}^{n}F^{i\bar{i}}(u_{si\bar{i}}u_{\bar{s}}+u_{\bar{s}i\bar{i}}u_{s})&=\varphi'\sum_{s=1}^{n}\left[\Big(f^{\frac{1}{k}}\Big)_{s}u_{\bar{s}}+\Big(f^{\frac{1}{k}}\Big)_{\bar{s}}u_{s}\right]+\varphi'\sum_{i,s,t=1}^{n}u_{\bar{s}}u_{t}F^{i\bar{i}}R_{i\bar{i}s\bar{t}}\nonumber\\
&\geq-2\varphi'|\nabla u|\left|\nabla\Big(f^{\frac{1}{k}}\Big)\right|-\gamma\varphi'|\nabla u|^{2}\operatorname{tr}F^{i\bar{j}}\nonumber\\
&\geq-\left|\nabla\Big(f^{\frac{1}{k}}\Big)\right|-\gamma\operatorname{tr}F^{i\bar{j}}.
\end{align}
Moreover,
\begin{align}\label{eqn3.8}
\varphi'\sum_{i=1}^{n}F^{i\bar{i}}u_{i\bar{i}}^{2}+\psi'\sum_{i=1}^{n}F^{i\bar{i}}u_{i\bar{i}}&=\varphi'\sum_{i=1}^{n}F^{i\bar{i}}\big(\lambda_{i}^{2}-2\lambda_{i}+1\big)+\psi'\sum_{i=1}^{n}F^{i\bar{i}}(\lambda_{i}-1)\nonumber\\
&=\varphi'\sum_{i=1}^{n}F^{i\bar{i}}\lambda_{i}^{2}+\big(\varphi'-\psi'\big)\operatorname{tr}F^{i\bar{j}}+\big(\psi'-2\varphi'\big)f^{\frac{1}{k}}.
\end{align}
Combining \eqref{eqn3.4}--\eqref{eqn3.8}, we obtain
\begin{align}\label{eqn3.9}
0&\geq-\sum_{i,j,s,t=1}^{n}\frac{F^{i\bar{j},s\bar{t}}u_{i\bar{j}1}u_{s\bar{t}\bar{1}}}{1+u_{1\bar{1}}}-\sum_{i=1}^{n}\frac{F^{i\bar{i}}|u_{1\bar{1}i}|^{2}}{(1+u_{1\bar{1}})^{2}}+\varphi''\sum_{i=1}^{n}F^{i\bar{i}}\left|u_{i}u_{i\bar{i}}+\sum_{s=1}^{n}u_{si}u_{\bar{s}}\right|^{2}\nonumber\\
&\hspace{4.5mm}+\psi''\sum_{i=1}^{n}F^{i\bar{i}}|u_{i}|^{2}+\varphi'\sum_{i=1}^{n}F^{i\bar{i}}\lambda_{i}^{2}+\big(-\psi'+\varphi'-2\gamma\big)\operatorname{tr}F^{i\bar{j}}+\mathbf{J},
\end{align}
where 
\begin{equation}\label{eqn3.10}
\mathbf{J}:=\frac{1}{1+u_{1\bar{1}}}\Big(f^{\frac{1}{k}}\Big)_{1\bar{1}}-\left|\nabla\Big(f^{\frac{1}{k}}\Big)\right|+\big(\psi'-2\varphi'-\gamma\big)f^{\frac{1}{k}}.
\end{equation}

Up to this point, we have followed Hou-Ma-Wu \cite{Hou2010}.
By our choice,
\begin{equation}\label{eqn3.11}
-\psi'+\varphi'-2\gamma\geq1.
\end{equation}
We note that the inequality $\psi''-\kappa\psi'^{2}\geq0$ is equivalent to 
\[
\bigg(\frac{1}{-\psi'}\bigg)'\geq\kappa,
\]
and hence one can only expect this to hold for small $\kappa>0$, since $-\psi'$ is large.
Hou-Ma-Wu \cite{Hou2010} used a very good case-by-case argument different from that of  Chou-Wang \cite{Wang2001}.
We next show that the case-by-case argument in \cite[Theorem 4.1]{Wang2001} is also applicable here.

\par
\vspace{2mm}
\textbf{Case 1:} $\lambda_{k}\leq\varepsilon\lambda_{1}$, where $\varepsilon>0$ is to be determined.
Using \eqref{eqn3.3} and \eqref{eqn3.1}, we estimate the combination of the second, third, and fifth terms in \eqref{eqn3.9} for $i=1$ as 
\begin{align}\label{eqn3.12}
&\hspace{4.5mm}-\frac{F^{1\bar{1}}|u_{1\bar{1}1}|^{2}}{(1+u_{1\bar{1}})^{2}}+\varphi''F^{1\bar{1}}\left|u_{1}u_{1\bar{1}}+\sum_{s=1}^{n}u_{s1}u_{\bar{s}}\right|^{2}+\varphi'F^{1\bar{1}}\lambda_{1}^{2}\nonumber\\
&=-F^{1\bar{1}}\left(\varphi'u_{1}u_{\bar{1}1}+\varphi'\sum_{s=1}^{n}u_{s1}u_{\bar{s}}+\psi'u_{1}\right)^{2}+\varphi''F^{1\bar{1}}\left|u_{1}u_{1\bar{1}}+\sum_{s=1}^{n}u_{s1}u_{\bar{s}}\right|^{2}+\varphi'F^{1\bar{1}}\lambda_{1}^{2}\nonumber\\
&\geq\big(\varphi''-2\varphi'^{2}\big)F^{1\bar{1}}\left|u_{1}u_{1\bar{1}}+\sum_{s=1}^{n}u_{s1}u_{\bar{s}}\right|^{2}+F^{1\bar{1}}\big(\varphi'\lambda_{1}^{2}-2\psi'^{2}|u_{1}|^{2}\big)\nonumber\\
&\geq F^{1\bar{1}}\bigg(\frac{1}{4\beta}\lambda_{1}^{2}-2\psi'^{2}\beta\bigg)\nonumber\\
&\geq0,
\end{align}
where the last inequality holds, since otherwise we are done.
Next, we control the second term in \eqref{eqn3.9} for $i=2,\dots,n$ by the first, third, and fourth terms in \eqref{eqn3.9}. 
Using \eqref{eqn3.3}, \eqref{eqn3.1}, \eqref{eqn2.2}, and the concavity of $\sigma_{k}^{1/k}(\lambda)$, we have
\begin{align}\label{eqn3.13}
&\hspace{4.5mm}-\sum_{i,j,s,t=1}^{n}\frac{F^{i\bar{j},s\bar{t}}u_{i\bar{j}1}u_{s\bar{t}\bar{1}}}{1+u_{1\bar{1}}}-(1-\delta)\sum_{i=2}^{n}\frac{F^{i\bar{i}}|u_{1\bar{1}i}|^{2}}{(1+u_{1\bar{1}})^{2}}\nonumber\\
&\hspace{4.5mm}-\delta\sum_{i=2}^{n}\frac{F^{i\bar{i}}|u_{1\bar{1}i}|^{2}}{(1+u_{1\bar{1}})^{2}}+\varphi''\sum_{i=2}^{n}F^{i\bar{i}}\left|u_{i}u_{i\bar{i}}+\sum_{s=1}^{n}u_{si}u_{\bar{s}}\right|^{2}+\psi''\sum_{i=1}^{n}F^{i\bar{i}}|u_{i}|^{2}\nonumber\\
&\geq-\sum_{i,j,s,t=1}^{n}\frac{F^{i\bar{j},s\bar{t}}u_{i\bar{j}1}u_{s\bar{t}\bar{1}}}{1+u_{1\bar{1}}}-(1-\delta)\sum_{i=2}^{n}\frac{F^{i\bar{i}}|u_{1\bar{1}i}|^{2}}{(1+u_{1\bar{1}})^{2}}\nonumber\\
&\hspace{4.5mm}+\big(\varphi''-2\delta\varphi'^{2}\big)\sum_{i=2}^{n}F^{i\bar{i}}\left|u_{i}u_{i\bar{i}}+\sum_{s=1}^{n}u_{si}u_{\bar{s}}\right|^{2}+\big(\psi''-2\delta\psi'^{2}\big)\sum_{i=2}^{n}F^{i\bar{i}}|u_{i}|^{2}\nonumber\\
&\geq-\sum_{i\neq s}\frac{F^{i\bar{s},s\bar{i}}u_{i\bar{s}1}u_{s\bar{i}\bar{1}}}{1+u_{1\bar{1}}}-(1-\delta)\sum_{i=2}^{n}\frac{F^{i\bar{i}}|u_{1\bar{1}i}|^{2}}{(1+u_{1\bar{1}})^{2}}\nonumber\\
&\geq-\sum_{i=2}^{n}\frac{F^{i\bar{1},1\bar{i}}|u_{1\bar{1}i}|^{2}}{1+u_{1\bar{1}}}-(1-\delta)\sum_{i=2}^{n}\frac{F^{i\bar{i}}|u_{1\bar{1}i}|^{2}}{(1+u_{1\bar{1}})^{2}}\nonumber\\
&=\frac{1}{\lambda_{1}}\frac{1}{k}\sigma_{k}^{1/k-1}\sum_{i=2}^{n}\left(\sigma_{k-2;i1}-(1-\delta)\frac{\sigma_{k-1;i}}{\lambda_{1}}\right)|u_{1\bar{1}i}|^{2}.
\end{align}
It follows from Chou-Wang \cite[Lemma 3.1]{Wang2001} that there exists a small constant $\varepsilon>0$, depending on $\delta$, such that 
\begin{equation}\label{eqn3.14}
\sigma_{k-2;i1}-(1-\delta)\frac{\sigma_{k-1;i}}{\lambda_{1}}\geq0\quad \text{for}\ i=2,\dots,n.
\end{equation}
It follows from \eqref{eqn3.9} and \eqref{eqn3.11}--\eqref{eqn3.14} that 
\begin{equation}\label{eqn3.15}
0\geq\operatorname{tr}F^{i\bar{j}}+\mathbf{J}.
\end{equation}

\par
\vspace{2mm}
\textbf{Case 2:} $\lambda_{k}>\varepsilon\lambda_{1}$.
In this case, the fifth term in \eqref{eqn3.9} is sufficient.
By Chou-Wang \cite[inequality (3.2)]{Wang2001}, we have 
\begin{equation}\label{eqn3.16}
\varphi'\sum_{i=1}^{n}F^{i\bar{i}}\lambda_{i}^{2}\geq\varphi'F^{kk}\lambda_{k}^{2}\geq\frac{1}{C}\varphi'\lambda_{k}^{2}\operatorname{tr}F^{i\bar{j}}\geq\frac{\varepsilon^{2}}{4C\beta}\lambda_{1}^{2}\operatorname{tr}F^{i\bar{j}},
\end{equation}
where $C>0$ is universal.
Using \eqref{eqn3.3} and \eqref{eqn3.1}, we estimate the second and third terms in \eqref{eqn3.9} as 
\begin{align}\label{eqn3.17}
&\hspace{4.5mm}-\sum_{i=1}^{n}\frac{F^{i\bar{i}}|u_{1\bar{1}i}|^{2}}{(1+u_{1\bar{1}})^{2}}+\varphi''\sum_{i=1}^{n}F^{i\bar{i}}\left|u_{i}u_{i\bar{i}}+\sum_{s=1}^{n}u_{si}u_{\bar{s}}\right|^{2}\nonumber\\
&=-\sum_{i=1}^{n}F^{i\bar{i}}\left(\varphi'u_{i}u_{\bar{i}i}+\varphi'\sum_{s=1}^{n}u_{si}u_{\bar{s}}+\psi'u_{i}\right)^{2}+\varphi''\sum_{i=1}^{n}F^{i\bar{i}}\left|u_{i}u_{i\bar{i}}+\sum_{s=1}^{n}u_{si}u_{\bar{s}}\right|^{2}\nonumber\\
&\geq\big(\varphi''-2\varphi'^{2}\big)\sum_{i=1}^{n}F^{i\bar{i}}\left|u_{i}u_{i\bar{i}}+\sum_{s=1}^{n}u_{si}u_{\bar{s}}\right|^{2}-2\psi'^{2}\sum_{i=1}^{n}F^{i\bar{i}}|u_{i}|^{2}\nonumber\\
&\geq-2\beta\psi'^{2}\operatorname{tr}F^{i\bar{j}}.
\end{align}
It follows from \eqref{eqn3.9}, \eqref{eqn3.11}, \eqref{eqn3.16}, \eqref{eqn3.17}, and the concavity of $F$ that 
\begin{equation}\label{eqn3.18}
0\geq\left(\frac{\varepsilon^{2}}{4C\beta}\lambda_{1}^{2}+1-2\beta\psi'^{2}\right)\operatorname{tr}F^{i\bar{j}}+\mathbf{J}\geq\operatorname{tr}F^{i\bar{j}}+\mathbf{J},
\end{equation}
where the last inequality holds, since otherwise we are done.

Finally, we estimate $\mathbf{J}$.
For convenience, set $\tilde{f}:=f^{3/(2k-2)}$.
It follows from \eqref{eqn3.10} that 
\begin{align*}
\mathbf{J}&\geq-\left|\Big(f^{\frac{1}{k}}\Big)_{1\bar{1}}\right|-\left|\nabla\Big(f^{\frac{1}{k}}\Big)\right|+\big(\psi'-2\varphi'-\gamma\big)f^{\frac{1}{k}}\nonumber\\
&=-\frac{2k-2}{3k}\tilde{f}^{\frac{-k-2}{3k}}\left|\tilde{f}_{1\bar{1}}-\frac{k+2}{3k}\frac{\big|\tilde{f}_{1}\big|^{2}}{\tilde{f}}\right|-\frac{2k-2}{3k}\tilde{f}^{\frac{-2}{3k}+\frac{1}{6}}\frac{\big|\nabla\tilde{f}\big|}{\tilde{f}^{1/2}}+\big(\psi'-2\varphi'-\gamma\big)f^{\frac{1}{k}}.
\end{align*}
Since $(k+2)/(3k)<1/2$ for $k\geq5$, Lemma \ref{thm3.1} yields 
\begin{equation}\label{eqn3.19}
\mathbf{J}\geq-K_{3}f^{-\frac{1}{k(k-1)}}+\big(\psi'-2\varphi'-\gamma\big)f^{\frac{1}{k}},
\end{equation}
where $K_{3}>0$ depends on $\|f^{3/(2k-2)}\|_{C^{2,1}(M)}$ and the geometry of $(M,\omega)$, but is independent of $\inf_{M}f$.
Recalling \eqref{eqn2.1}, 
\begin{equation}\label{eqn3.20}
\operatorname{tr}F^{i\bar{j}}\geq C\sigma_{1}^{\frac{1}{k-1}}f^{-\frac{1}{k(k-1)}},
\end{equation}
where $C>0$ is universal.
Combining \eqref{eqn3.15}, \eqref{eqn3.18}, \eqref{eqn3.19}, and \eqref{eqn3.20}, we complete the proof of Theorem \ref{thm1.1}. 
\end{proof}

\section{Christoffel-Minkowski equation}

We establish the following lemma for real-valued functions on $\mathbb{S}^{n}$, whose proof is postponed to the Appendix. 

\begin{lemma}\label{thm4.1}
Let $\tilde{f}$ be a positive function on $\mathbb{S}^{n}$ with $\tilde{f}\in C^{2,1}(\mathbb{S}^{n})$.
Then for any $\alpha<1/2$, 
\[
\nabla_{ee}^{2}\tilde{f}(x)-\alpha\frac{\big|\nabla_{e}\tilde{f}(x)\big|^{2}}{\tilde{f}(x)}\geq-K\tilde{f}^{\frac{1}{3}}(x),\quad \forall x\in\mathbb{S}^{n},\ \forall e\in T_{x}\mathbb{S}^{n},\ |e|=1,
\]
where $K>0$ depends on $\alpha$ and $\big\|\tilde{f}\big\|_{C^{2,1}(\mathbb{S}^{n})}$, but is independent of $\inf_{\mathbb{S}^{n}}\tilde{f}$.
\end{lemma}

For convenience, let $W=(W_{ij})$ with $W_{ij}:=u_{ij}+u\delta_{ij}$.
In the following lemma, we derive a bound on the eigenvalues of $W$.

\begin{lemma}\label{thm4.2}
Let $u\in C^{4}(\mathbb{S}^{n})$ be a $k$-admissible solution of \eqref{eqn1.7}, and $\inf_{\mathbb{S}^{n}}f>0$.
Then  
\[
|\lambda(W)|\leq C\quad \text{on}\ \mathbb{S}^{n},
\]
where $C>0$ depends on $\|f^{3/(2k-2)}\|_{C^{2,1}(\mathbb{S}^{n})}$, but is independent of $\inf_{\mathbb{S}^{n}}f$.
\end{lemma}
\begin{proof}
Construct the auxiliary function
\[
H:=\sigma_{1}(W)=\Delta u+nu\quad \text{on}\ \mathbb{S}^{n},
\]
where $\Delta u=\sum_{i=1}^{n}\nabla_{ii}^{2}u$.
Let $x_{0}\in\mathbb{S}^{n}$ be a point where $H$ attains its maximum.
Choose an orthonormal local frame $\{e_{i}\}_{i=1}^{n}$ near $x_{0}$ such that 
$u_{ij}(x_{0})=\delta_{ij}u_{ii}(x_{0})$.
Then both $W(x_{0})$ and $(\sigma_{k}^{ij}(W(x_{0})))$ are diagonal.
We have at $x_{0}$,
\begin{equation}\label{eqn4.1}
\nabla_{ii}^{2}H=\Delta(W_{ii})-nW_{ii}+H,\quad \forall1\leq i\leq n.
\end{equation}
Since $H$ attains its maximum at $x_{0}$, we have at $x_{0}$,
\begin{equation}\label{eqn4.2}
0=\nabla_{i}H,\quad \forall1\leq i\leq n,
\end{equation}
and 
\begin{align}\label{eqn4.3}
0\geq\sum_{i=1}^{n}\sigma_{k}^{ii}(W)\nabla_{ii}^{2}H&=\sum_{i=1}^{n}\sigma_{k}^{ii}(W)\Delta(W_{ii})-\sum_{i=1}^{n}n\sigma_{k}^{ii}(W)W_{ii}+H\operatorname{tr}\sigma_{k}^{ij}(W)\nonumber\\
&=\sum_{i=1}^{n}\sigma_{k}^{ii}(W)\Delta(W_{ii})-nkf+H\operatorname{tr}\sigma_{k}^{ij}(W).
\end{align}
Applying the Laplacian to $\sigma_{k}(W)=f$, we obtain
\begin{equation}\label{eqn4.4}
\sum_{i,j=1}^{n}\sigma_{k}^{ij}(W)\Delta(W_{ij})=\Delta f-\sum_{i,j,s,t,l=1}^{n}\sigma_{k}^{ij,st}(W)\nabla_{l}(W_{ij})\nabla_{l}(W_{st}).
\end{equation}
It follows from Lemma \ref{thm2.2} and \eqref{eqn4.2} that at $x_{0}$, 
\begin{equation}\label{eqn4.6}
-\sum_{i,j,s,t,l=1}^{n}\sigma_{k}^{ij,st}(W)\nabla_{l}(W_{ij})\nabla_{l}(W_{st})\geq -\sum_{l=1}^{n}\frac{k-2}{k-1}\frac{|\nabla_{l}f|^{2}}{f}.
\end{equation}
Therefore, by \eqref{eqn4.3}, \eqref{eqn4.4}, and \eqref{eqn4.6}, 
\begin{equation}\label{eqn4.7}
0\geq\Delta f-\sum_{l=1}^{n}\frac{k-2}{k-1}\frac{|\nabla_{l}f|^{2}}{f}-nkf+H\operatorname{tr}\sigma_{k}^{ij}(W).
\end{equation}

By Lemma \ref{thm2.1}, we have 
\begin{equation}\label{eqn4.8}
\operatorname{tr}\sigma_{k}^{ij}(W)=\sum_{i=1}^{n}\sigma_{k-1;i}(\lambda(W))\geq CH^{\frac{1}{k-1}}f^{1-\frac{1}{k-1}},
\end{equation}
where $C>0$ is universal.
For convenience, set
\[
\tilde{f}:=f^{\frac{3}{2k-2}}\quad \text{and}\quad p:=\frac{2k-2}{3}.
\]
Using Lemma \ref{thm4.1}, we derive
\begin{equation}\label{eqn4.9}
\Delta f-\sum_{l=1}^{n}\frac{k-2}{k-1}\frac{|\nabla_{l}f|^{2}}{f}=p\tilde{f}^{p-1}\sum_{l=1}^{n}\left(\nabla_{ll}^{2}\tilde{f}-\frac{1}{3}\frac{\big|\nabla_{l}\tilde{f}\big|^{2}}{\tilde{f}}\right)\geq-Kf^{1-\frac{1}{k-1}},
\end{equation}
where $K>0$ depends on $\|f^{3/(2k-2)}\|_{C^{2,1}(\mathbb{S}^{n})}$.
Combining \eqref{eqn4.7}--\eqref{eqn4.9}, we obtain 
\[
0\geq\Big(CH^{\frac{k}{k-1}}-K\Big)f^{1-\frac{1}{k-1}}-nkf
\]
at $x_{0}$.
Therefore, 
\[
\sup_{\mathbb{S}^{n}}\sigma_{1}(W)=H(x_{0})\leq\widetilde{K},
\]
where $\widetilde{K}>0$ depends on $\|f^{3/(2k-2)}\|_{C^{2,1}(\mathbb{S}^{n})}$, but is independent of $\inf_{\mathbb{S}^{n}}f$.
Since $\lambda(W)\in\Gamma_{2}$, we have
\[
|\lambda(W)|\leq\sigma_{1}(W)\leq\widetilde{K}\quad \text{on}\ \mathbb{S}^{n}.
\]
This completes the proof of Lemma \ref{thm4.2}.
\end{proof}

We now turn to the $C^{0}$ estimate for $u$.
If $u$ solves \eqref{eqn1.7}, then $u+\ell$ is also a solution for any linear function $\ell$.
To obtain a $C^{0}$ estimate, we impose the following orthogonality condition:
\begin{equation}\label{eqn4.10}
\int_{\mathbb{S}^{n}}x_{i}u=0\quad \text{for}\ i=1,2,\dots,n+1.
\end{equation}
The following lemma is due to Guan-Zhang \cite[Proposition 3.2]{Guan2021}; for completeness, we include its proof in the Appendix.

\begin{lemma}[\cite{Guan2021}]\label{thm4.3}
Let $u\in C^{4}(\mathbb{S}^{n})$ be a $k$-admissible solution of \eqref{eqn1.7} with $\inf_{\mathbb{S}^{n}}f>0$, and assume that $u$ satisfies \eqref{eqn4.10}.
Then 
\[
\|u\|_{C^{0}(\mathbb{S}^{n})}\leq C,
\]
where $C>0$ depends on $\|f^{3/(2k-2)}\|_{C^{2,1}(\mathbb{S}^{n})}$, but is independent of $\inf_{\mathbb{S}^{n}}f$.
\end{lemma}

Combining Lemmas \ref{thm4.2} and \ref{thm4.3}, we obtain a $C^{2}$ bound for $u$ on $\mathbb{S}^{n}$ that is independent of $\inf_{\mathbb{S}^{n}}f$.
The existence of a $k$-admissible solution in $C^{1,1}(\mathbb{S}^{n})$ then follows by an approximation argument via solutions of the non-degenerate problems; this yields Theorem \ref{thm1.3}.

\section{Appendix}

\begin{proof}[Proof of Lemma \ref{thm4.1}]
Fix a point $x_{0}\in\mathbb{S}^{n}$ and a unit vector $e\in T_{x_{0}}\mathbb{S}^{n}$.
Let $\gamma(t)$ be the unit-speed geodesic on $\mathbb{S}^{n}$ with $\gamma(0)=x_{0}$ and $\gamma'(0)=e$, and define
\[
h(t):=\tilde{f}(\gamma(t)),\qquad t\in[-\varepsilon,\varepsilon].
\]

Then
\[
h(0)=\tilde{f}(x_{0})>0,\quad h'(0)=\nabla_{e}\tilde{f}(x_{0}),\quad h''(0)=\nabla^{2}_{ee}\tilde{f}(x_{0}),
\]
and
\[
\|h\|_{C^{2,1}([-\varepsilon,\varepsilon])}\leq C\big\|\tilde{f}\big\|_{C^{2,1}(\mathbb{S}^{n})},
\]
where $C$ depends only on $n$.
Finally, applying Lemma \ref{thm2.3} to $h$ and returning to $\tilde{f}$, we complete the proof of Lemma \ref{thm4.1}.
\end{proof}

\par
\vspace{2mm}
\begin{proof}[Proof of Lemma \ref{thm3.1}]
Fix a point $x_{0}\in M$.
Since $(M,\omega)$ is K\"ahler, there exist local holomorphic coordinates $(z^{1},\dots,z^{n})$ centered at $x_{0}$ such that 
\[
g_{i\bar{j}}(x_{0})=\delta_{ij},\quad \frac{\partial g_{i\bar{j}}}{\partial z^{l}}(x_{0})=0,\quad \forall l,
\]
where $g_{i\bar{j}}$ are the components of the metric $\omega$ in these coordinates. 
In particular, at $x_{0}$ the Levi-Civita connection coincides with the coordinate derivatives
\[
\nabla_{i}=\partial_{i},\quad \nabla_{\bar{i}}=\partial_{\bar{i}}.
\]
Let $e\in T_{x_{0}}M$ be a unit vector with respect to $\omega$.
After a unitary change of coordinates, we may assume that
\[
e=\frac{\partial}{\partial z^{1}}\Big|_{x_{0}}.
\]

Now consider the two real directions corresponding to the real and imaginary parts of $e$.
Write $z^{1}=x^{1}+iy^{1}$ and define
\[
u_{1}:=\frac{1}{\sqrt{2}}\frac{\partial}{\partial x^{1}},\quad u_{2}:=\frac{1}{\sqrt{2}}\frac{\partial}{\partial y^{1}}.
\]
Choose $\varepsilon>0$ sufficiently small, depending only on the geometry of $(M,\omega)$, such that the geodesics
\[
\gamma_{1}(t):=\exp_{x_{0}}(t u_{1}),\quad \gamma_{2}(t):=\exp_{x_{0}}(t u_{2}),\qquad t\in[-\varepsilon,\varepsilon],
\]
are well-defined and remain in a coordinate neighborhood of $x_{0}$.
Such an $\varepsilon$ exists by the compactness of $M$.
Define
\[
h_{1}(t):=\tilde{f}(\gamma_{1}(t)),\quad h_{2}(t):=\tilde{f}(\gamma_{2}(t)),\qquad t\in[-\varepsilon,\varepsilon].
\]
Since $\tilde{f}\in C^{2,1}(M)$ and the exponential map is smooth, we have $h_{1},h_{2}\in C^{2,1}([-\varepsilon,\varepsilon])$ and $h_{1},h_{2}>0$.
Moreover, there exists $C=C(M,\omega)>0$ such that
\[
\max\left\{\|h_1\|_{C^{2,1}([-\varepsilon,\varepsilon])}, \|h_2\|_{C^{2,1}([-\varepsilon,\varepsilon])}\right\}\leq C\big\|\tilde{f}\big\|_{C^{2,1}(M)}.
\]

Applying Lemma \ref{thm2.3} to $h_{1}$ and $h_{2}$, we obtain that for any $\alpha<1/2$,
\begin{equation}\label{eqn4.11}
h_{i}''(0)-\alpha\frac{|h_{i}'(0)|^2}{h_{i}(0)}\geq-K_{i}h_{i}^{\frac{1}{3}}(0),\quad \text{for}\ i=1,2,
\end{equation}
where $K_{i}>0$ depends on $\alpha$, $\|h_{i}\|_{C^{2,1}([-\varepsilon,\varepsilon])}$, and $\varepsilon$, but is independent of $\inf_{[-\varepsilon,\varepsilon]}h_{i}$.
In the following, we compute the derivatives of $h_{1}$ and $h_{2}$ at $0$ in terms of $\tilde{f}$.
At $x_{0}$, we have
\begin{align*}
&h_{1}(0)=\tilde{f}(x_{0}),\quad
h_{1}'(0)=\frac{1}{\sqrt{2}}\frac{\partial\tilde{f}}{\partial x^{1}}(x_{0}),\quad
h_{1}''(0)=\frac{1}{2}\frac{\partial^{2}\tilde{f}}{(\partial x^{1})^{2}}(x_{0}),\\
&h_{2}(0)=\tilde{f}(x_{0}),\quad
h_{2}'(0)=\frac{1}{\sqrt{2}}\frac{\partial\tilde{f}}{\partial y^{1}}(x_{0}),\quad
h_{2}''(0)=\frac{1}{2}\frac{\partial^{2}\tilde{f}}{(\partial y^{1})^{2}}(x_{0}).
\end{align*}
Substituting these expressions into \eqref{eqn4.11}, we get at $x_{0}$,
\begin{equation}\label{eqn4.12}
\left[\frac{\partial^{2}\tilde{f}}{(\partial x^{1})^{2}}+\frac{\partial^{2}\tilde{f}}{(\partial y^{1})^{2}}\right]
-\alpha\frac{\big(\frac{\partial\tilde{f}}{\partial x^{1}}\big)^{2}+\big(\frac{\partial\tilde{f}}{\partial y^{1}}\big)^{2}}{\tilde{f}}
\geq-K\tilde{f}^{\frac{1}{3}},
\end{equation}
where $K>0$ depends on $\alpha$, $\big\|\tilde{f}\big\|_{C^{2,1}(M)}$, and the geometry of $(M,\omega)$, but is independent of $\inf_{M}\tilde{f}$.
Moreover, since the metric is K\"ahler and the coordinates are normal, the following identities hold at $x_{0}$:
\begin{equation}\label{eqn4.13}
\nabla_{e}\nabla_{\bar{e}}\tilde{f}=\frac{1}{4}\left[\frac{\partial^{2}\tilde{f}}{(\partial x^{1})^{2}}+\frac{\partial^{2}\tilde{f}}{(\partial y^{1})^{2}}\right],\quad  \big|\nabla_{e}\tilde{f}\big|^{2}=\frac{1}{4}\left[\bigg(\frac{\partial\tilde{f}}{\partial x^{1}}\bigg)^{2}+\bigg(\frac{\partial\tilde{f}}{\partial y^{1}}\bigg)^{2}\right].
\end{equation}
Combining \eqref{eqn4.12} and \eqref{eqn4.13}, we obtain 
\[
\nabla_{e}\nabla_{\bar{e}}\tilde{f}(x_{0})-\alpha\frac{\big|\nabla_{e}\tilde{f}(x_{0})\big|^{2}}{\tilde{f}(x_{0})}\geq-\frac{K}{4}\tilde{f}^{\frac{1}{3}}(x_{0}).
\]
Finally, since $x_{0}$ and $e$ were arbitrary, the estimate holds for all $x\in M$ and all unit vectors $e\in T_{x}M$.
This completes the proof.
\end{proof}

\par
\vspace{2mm}
\begin{proof}[Proof of Lemma \ref{thm4.3}]
The proof is a blow-up argument.
Suppose that no such bound holds.
Then there exist a sequence of functions $\{u^{m}\}_{m=1}^{\infty}$ and a constant $\widetilde{C}>0$, independent of $m$, such that  
\[
\sigma_{k}\big(u^{m}_{ij}+u^{m}\delta_{ij}\big)=f^{m}\quad \text{on}\ \mathbb{S}^{n},
\]
where $f^{m}\geq0$ satisfies
\[
\left\|(f^{m})^{3/(2k-2)}\right\|_{C^{2,1}(\mathbb{S}^{n})}\leq\widetilde{C},
\]
but
\[
\left\|u^{m}\right\|_{L^{\infty}(\mathbb{S}^{n})}\geq m.
\]

Define 
\[
v^{m}:=\frac{u^{m}}{\|u^{m}\|_{L^{\infty}(\mathbb{S}^{n})}}.
\]
Then, for each $m=1,2,\dots$,
\begin{equation}\label{eqn4.14}
\left\|v^{m}\right\|_{L^{\infty}(\mathbb{S}^{n})}=1.
\end{equation}
By Lemma \ref{thm4.2}, we have
\[
|\lambda(W_{u^{m}})|\leq C,
\]
where $C>0$ is independent of $m$.
Consequently,
\begin{equation}\label{eqn4.15}
|\lambda(W_{v^{m}})|\leq\frac{C}{\|u^{m}\|_{L^{\infty}(\mathbb{S}^{n})}}\rightarrow0.
\end{equation}
In particular, we obtain
\[
\Delta v^{m}+nv^{m}\rightarrow0.
\]
It follows from \eqref{eqn4.14} and \eqref{eqn4.15} that $D^{2}v^{m}$ is uniformly bounded; hence, by interpolation,
\[
\|v^{m}\|_{C^{2}(\mathbb{S}^{n})}\leq\overline{C}
\]
for some $\overline{C}$ independent of $m$.
Therefore, after passing to a subsequence $\{v^{m_{i}}\}_{i=1}^{\infty}$, there exist $\alpha\in(0,1)$ and a function $v\in C^{1,\alpha}(\mathbb{S}^{n})$ satisfying \eqref{eqn4.10} such that
\[
v^{m_{i}}\rightarrow v\quad \text{in}\ C^{1,\alpha}(\mathbb{S}^{n})\quad \text{and}\quad \|v\|_{L^{\infty}(\mathbb{S}^{n})}=1.
\]
Moreover, 
\[
\Delta v+nv=0\quad \text{on}\ \mathbb{S}^{n}
\]
in the distributional sense.
By linear elliptic theory, $v$ is in fact smooth.
Since $v$ satisfies the orthogonality condition \eqref{eqn4.10}, we conclude that $v\equiv0$ on $\mathbb{S}^{n}$, which contradicts $\|v\|_{L^{\infty}(\mathbb{S}^{n})}=1$.
This completes the proof of Lemma \ref{thm4.3}.
\end{proof}

\end{document}